\documentclass[a4paper,11pt]{article}

\usepackage{amsmath}
\usepackage{amsthm}
\usepackage{mathrsfs}
\usepackage{latexsym}
\usepackage{amssymb}
\usepackage{amscd}
\usepackage[dvips]{graphics}
\usepackage[all,cmtip]{xy}
\usepackage{enumerate}
\usepackage[colorlinks=true]{hyperref}
\usepackage{tikz-cd}
\usepackage{enumitem}
\usepackage{fullpage}

\usepackage{palatino}
\usepackage{authblk}

\hypersetup{colorlinks=true,linkcolor=[rgb]{0.5, 0.0, 0},citecolor=[rgb]{0.5, 0.0, 0.5}, filecolor=magenta,urlcolor=blue}

\DeclareSymbolFont{AMSb}{U}{msb}{m}{n}
\DeclareMathSymbol{\N}{\mathbin}{AMSb}{"4E}
\DeclareMathSymbol{\Z}{\mathbin}{AMSb}{"5A}
\DeclareMathSymbol{\R}{\mathbin}{AMSb}{"52}
\DeclareMathSymbol{\Q}{\mathbin}{AMSb}{"51}
\DeclareMathSymbol{\I}{\mathbin}{AMSb}{"49}
\DeclareMathSymbol{\C}{\mathbin}{AMSb}{"43}
\DeclareMathOperator{\cchar}{char}

\DeclareFontFamily{U}{mathx}{\hyphenchar\font45}
\DeclareFontShape{U}{mathx}{m}{n}{
      <5> <6> <7> <8> <9> <10>
      <10.95> <12> <14.4> <17.28> <20.74> <24.88>
      mathx10
      }{}
\DeclareSymbolFont{mathx}{U}{mathx}{m}{n}
\DeclareFontSubstitution{U}{mathx}{m}{n}
\DeclareMathAccent{\widecheck}{0}{mathx}{"71}
\DeclareMathAccent{\wideparen}{0}{mathx}{"75}

\newcommand{\dbl}{[\hspace{-0.2ex}[}
\newcommand{\dbr}{]\hspace{-0.2ex}]}
\newcommand{\db}[1]{\dbl {#1} \dbr}

\newcommand{\ctens}{\widehat{\otimes}}

\newcommand{\iso}{\cong}
\newcommand{\invlim}{\underleftarrow{\textnormal{lim}}\,}
\newcommand{\dirlim}{\underrightarrow{\textnormal{lim}}\,}

\newcommand{\onto}{\twoheadrightarrow}

\newcommand{\ct}{\widehat\otimes}

\newcommand{\Ker}{\textnormal{Ker}}

\newcommand{\tn}[1]{\textnormal{#1}}

\newcommand{\F}{\mathbb{F}}

\numberwithin{equation}{section}

\title{Semisimplicity and separability for pseudocompact algebras}

\author[a]{Kostiantyn Iusenko}
\author[b]{John William MacQuarrie}
\affil[a]{Instituto de Matem\'{a}tica e Estat\'{i}stica, Univ. de S\~{a}o Paulo, S\~{a}o Paulo, SP, Brazil}
\affil[b]{Universidade Federal de Minas Gerais, Belo Horizonte, MG, Brazil}

\date{}

\newtheorem{prop}[equation]{Proposition}
\newtheorem{lemma}[equation]{Lemma}
\newtheorem{theorem}[equation]{Theorem}
\newtheorem{corol}[equation]{Corollary}

\theoremstyle{remark}
\newtheorem{rem}[equation]{Remark}
\theoremstyle{definition}
\newtheorem{defn}[equation]{Def{i}nition}
\newtheorem{ex}[equation]{Example}

\begin{document}

\footnotetext{\textit{Email addresses:} iusenko@ime.usp.br (Kostiantyn Iusenko), john@mat.ufmg.br (John MacQuarrie)}

\maketitle

\begin{abstract}
We give a self-contained introduction to the wonderfully well-behaved class of pseudocompact algebras, focusing on the foundational classes of semisimple and separable algebras.  We give characterizations of such algebras analogous to those for finite dimensional algebras.  We give a self-contained proof of the Wedderburn-Malcev Theorem for pseudocompact algebras.
%In this short manuscript we study semisimple and separable  algebras extending the well-known results from Artinian to pseudocompact case. All the results are presented in self-enclosed form.

\

\end{abstract}

\section{Introduction}

Let $k$ be a field.  A pseudocompact $k$-algebra is the inverse limit of an inverse system of finite dimensional associative unital $k$-algebras, taken in the category of topological $k$-algebras.  Pseudocompact algebras thus have the same role in the world of associative algebras as profinite groups do in the world of groups.  They appear naturally in many contexts: for instance, if one wishes to study the representation theory of the profinite group $G$, one must study the modules for the completed group algebra $k\db{G}$, defined to be the inverse limit of the finite dimensional group algebras $k[G/N]$ where $N$ runs through the open normal subgroups of $G$.  As another example, the category of pseudocompact algebras and continuous algebra homomorphisms is precisely dual to the category of coalgebras and coalgebra homomorphisms -- a category central to the study of Hopf algebras.

Pseudocompact algebras are incredibly well-behaved, with a theory that in many ways resembles that of finite dimensional algebras.  This said, pseudocompact algebras have an unfair reputation for being difficult and technical: a slanderous claim we intend to set right!  Part of the problem seems to stem from the fact that the literature on pseudocompact algebras is widely scattered: you might only find the simple result you want about pseudocompact algebras as a special case of a complicated result stated for some wider and more difficult class of algebras, for example.

Here we intend to give a clear and uncomplicated survey of some foundational results in pseudocompact algebras, assuming little more than a working knowledge of basic results of finite dimensional associative algebras, and some undergraduate topology.  We focus on the fundamental class of semisimple pseudocompact algebras, and the even better behaved subclass of separable pseudocompact algebras.  We will note that, exactly as with finite dimensional algebras, semisimple algebras are a powerful starting place from which to study arbitrary pseudocompact algebras.

The text is organized as follows. In Section \ref{Section Prelims} we collect the definitions and basic properties we will need to study pseudocompact algebras. In Section \ref{section Jacobson radical and ss} we first discuss the topological Jacobson radical of a pseudocompact algebra, giving it a robust characterization in Proposition \ref{radicalChar}.  We then define the notion of semisimpicity for pseudocompact algebras and characterize these algebras in Proposition \ref{ssCharacterization}.  In Section \ref{Section Separable and WM} we consider separable pseudocompact algebras, characterizing them in Theorem \ref{Theorem separable characterization}, before finishing with a self-contained proof of the Wedderburn-Malcev Theorem for pseudocompact algebras (Theorems \ref{theorem Wedderburn splitting} and \ref{Malcev Uniqueness}).

Full disclosure regarding the phrase ``self-contained'': while the main thread and principal results of the discussion are for the most part honest-to-goodness self-contained, we allow ourselves a little more freedom to refer to external literature in two circumstances.  Firstly in examples: frequently examples can be illuminating even without a rigorous understanding of the justifications, and for this reason we include them without going into details, citing external results for proofs and further details.  Secondly, our focus is on semisimple and separable algebras, rather than a survey of pseudocompact algebras in general.  There are concepts (basic properties of the inverse limit functor, existence of a tensor product, etc.) whose proofs do not add to the discussion, and in these cases we again allow ourselves to cite the literature.

\section{Preliminaries}\label{Section Prelims}

Throughout, let $k$ be a field, treated as a discrete topological ring.

\begin{defn}
A \emph{pseudocompact $k$-algebra} is a topological $k$-algebra having basis of $0$ consisting of ideals of finite codimension whose intersection is $0$, and complete with respect to this topology.  Equivalently, a pseudocompact $k$-algebra is an inverse limit of discrete finite dimensional $k$-algebras and algebra homomorphisms.

If $A$ is a pseudocompact algebra, say that an $A$-module $U$ is \emph{discrete} if its topology is discrete and the multiplication $A\times U\to U$ is continuous (where $A\times U$ is given the product topology).  A \emph{pseudocompact $A$-module} is an inverse limit of discrete finite dimensional $A$-modules. The
category of pseudocompact $A$-modules is an abelian category with exact
inverse limits, see, for instance, \cite[Chapter IV, Theorem 3]{Gabriel_thesis}.
\end{defn}

We provide several examples.

\medskip

\begin{ex}  Let $A$ be the algebra of formal power series $k\db{x}$ in the variable $x$.  Every ideal of $k\db{x}$ apart from $0$ is of the form $(x^n)$ for some $n\in \N$, which has codimension $n$ in $A$.  For each $n$, the natural map from $k\db{x}/(x^{n+1}) = k[x]/(x^{n+1})$ to $k\db{x}/(x^n)$ sending $x$ to $x$ is a surjective algebra homomorphism.  We obtain by composing these maps an inverse system of finite dimensional algebras as follows:
$$\hdots \to k\db{x}/(x^3) \to k\db{x}/(x^2) \to k\db{x}/(x).$$
The inverse limit of this inverse system is $k\db{x}$, and so $k\db{x}$ is pseudocompact.  More generally, the ring of power series in any number of variables (finite or infinite) is a pseudocompact algebra.
\end{ex}

\medskip

\begin{ex} A \emph{profinite group} is a compact, Hausdorff, totally disconnected topological group.  In other words, a profinite group is an inverse limit in the category of topological groups of an inverse system of discrete finite groups.  Profinite groups are most frequently encountered in nature as Galois groups of Galois field extensions: If $L$ is a Galois extension of a field $K$ then $L$ is the union (= direct limit) of the finite intermediate Galois extensions of $L/K$.  If $F$ is such an extension, then $\tn{Gal}(F:K)$ is
a finite group.  Furthermore, an inclusion $F\to F'$ induces a surjective group homomorphism $\tn{Gal}(F':K)\to \tn{Gal}(F:K)$ by restricting the domain and codomain of each $\rho:F'\to F'$ to $F$.  Thus, applying $\tn{Gal}(-:K)$, the direct system of field extensions becomes an inverse system of finite groups and we define the Galois group $\tn{Gal}(L:K)$ of $L/K$ to be the inverse limit of this inverse system.  The topology of a profinite group (and by extension, of the algebra we will shortly define) is justified by the generalization of the Fundamental Theorem of Galois Theory to this context: there is a perfect order reversing correspondence between the intermediate extensions of $L/K$ and the \emph{closed} subgroups of the profinite group $\tn{Gal}(L:K)$ (see for instance \cite[Theorem 2.11.3]{RibZal}).  If $G$ is a profinite group and $k$ is a field, then for each continuous finite quotient $G/N$ we may consider the finite dimensional group algebra $k[G/N]$.  The inverse system of finite groups induces an inverse system of finite dimensional algebras in the obvious way, and we define the \emph{completed group algebra} $k\db{G}$ to be the inverse limit of this inverse system.  Thus $k\db{G}$ is a pseudocompact algebra by definition and it is the natural place to study the representation theory of the profinite group $G$ \cite{Brumer, MacQmodreps, MacQgreencorr, MacQSymondsbrauertheory, Symondspermcomplexes, Symondsdoublecoset, MelnikovSubgpsHomology}.
\end{ex}

\medskip

\begin{ex} A \emph{quiver} $Q$ is a directed graph, with multiple arrows and cycles permitted.  From $Q$ we can construct a pseudocompact algebra $k\db{Q}$, the \emph{completed path algebra} of $Q$, in analogy with the well-known construction in finite dimensional algebras.  The details are not difficult but require some further definitions that would lead us away from our main interest: for the construction in the finite case see \cite[Chapter III.1]{ARS} and in the pseudocompact case see \cite[Section 8] {SimsonCoalg}.  As a simple example,  the completed path algebra $k\db{Q}$ of the quiver having one vertex and one loop $x$, is the algebra of formal power series $k\db{x}$ discussed above.  As in the finite dimensional case, completed path algebras are fundamental examples for pseudocompact algebras: if $k$ is algebraically closed and $A$ is a pseudocompact $k$-algebra, then the category of pseudocompact $A$-modules is equivalent to the category of pseudocompact modules for the algebra $k\db{Q}/I$, where $Q$ is a quiver and $I$ is a suitable closed ideal of $k\db{Q}$ (this follows by duality from  \cite[Theorem 4.3]{CMo}).
\end{ex}
 
%A path in $Q$ is a meaningful concatenation of arrows (for instance we can compose arrows $\alpha$ and $\beta$ if, and only if, $\beta$ starts where $\alpha$ ends).  The length of a path is the number of arrows it includes.  There is at each vertex a unique path of length $0$.  From this data one may construct a pseudocompact algebra, the \emph{completed path algebra} $k\db{Q}$ of $Q$ as follows:  denote by $[Q_i]$ the pseudocompact vector space with (topological)    If $k$ is algebraically closed, a fundamental approach to the representation theory of finite dimensional algebras comes from the observation due to Gabriel [REF] that every 

\medskip

\begin{ex} \label{Ex.products}
For us, an important class of examples of infinite dimensional pseudocompact algebras and modules comes by taking direct products of finite dimensional objects of the same type. Namely, if $\left\{X_{i}, i\in I\right\}$ is a collection of finite dimensional topological $k$-vector spaces indexed by a set $I$ (each possibly with the structure of an algebra or a module), the corresponding direct product is the topological $k$-vector space $X:=\prod_{i \in I} X_{i}$, endowed with the product topology.  Additional structure is applied coordinate-wise, making the direct product into a pseudocompact object of the same type. 
One may check that $X$ can be expressed as the inverse limit of an inverse system of finite direct products as follows.  Given any subset $F$ of $I$, denote by $X_F$ the direct product $\prod_{i \in F} X_{i}$ (so that in particular $X=X_I$).  Let $\mathcal F$ be the set of all finite subsets $F$ of $I$, and consider the inverse system $\{X_F, \varphi_{FG},\mathcal F\}$ where $\varphi_{FG}:X_G \to X_F$, defined whenever $F\subseteq G$, is the obvious projection. Then
$$
   X=\invlim_{F\in \mathcal F} X_F.
$$
While not every pseudocompact object is a product of finite dimensional objects, the product notation can still be utilized for a general pseudocompact object, for the following reason: if $X$ is the inverse limit of the inverse system $\invlim_{i\in I}\{X_i, \varphi_{ij}\}$ of finite dimensional objects, then
$$X \iso \bigg\{(x_i)_{i\in I}\in \prod_{i\in I}X_i\,|\,\varphi_{ij}(x_j) = x_i,\  i\leqslant j\bigg\},$$
a closed subobject of $X_I$.  We do not use this (standard) fact explicitly in this discussion, so we skip the details.
% On the other hand, given an inverse system $\left\{X_{i}  , \varphi_{i j}, I\right\}$ of finite dimensional pseudocompact $A$-modules (respectively, finite dimensional $k$-algebras), one can understand its inverse limit $\invlim _{i \in I} X_{i}$ as a closed subspace $X$ of $X_I$ consisting of those tuples $\left(x_{i}\right)$ that satisfy the condition $\varphi_{i j}\left(x_{i}\right)=x_{j}$ if $i \succeq j$. If $\varphi_{i}: X \longrightarrow X_{i}$
% denotes the restriction of the canonical projection $X_I \longrightarrow X_{i}$, then one easily checks that each $\varphi_{i}$ is  a continuous homomorphism and that $\left(X, \varphi_{i}\right)$ is an inverse limit $\invlim _{i \in I} X_{i}$. In particular, we have that any pseudocompact module $X$ (respectively any pseudocompact $k$-algebra) can be viewed as a closed submodule (resp., subalgebra) in $X_I$ with each $X_i$ being finite dimensional pseudocompact $A$-module (respectively, finite dimensional $k$-algebra).
\end{ex}

\medskip

\begin{ex} A \emph{simple} pseudocompact module for the pseudocompact algebra $A$ is a non-zero pseudocompact module not having any closed submodules other than $0$ and itself.  Any infinite dimensional pseudocompact $A$-module has a lot of submodules by definition, so that in particular simple pseudocompact $A$-modules are finite dimensional.  Finite dimensional pseudocompact $A$-modules are necessarily discrete, and hence simple pseudocompact $A$-modules are simple as abstract $A$-modules.
%It is not true that a discrete finite dimensional $A$-module is necessarily pseudocompact, because the multiplication map might not be continuous.  

It is \emph{not} the case that an arbitrary discrete finite dimensional $A$-module is necessarily pseudocompact!  This is because not every left ideal of finite codimension is open in $A$.  To exhibit an explicit example is tricky, but one may work indirectly as follows: Let $\F_p$ be the field of $p$ elements, and consider $A = \prod_{n\in \N}\F_p$, a countable product of copies of the 1-dimensional algebra $\F_p$.  The ideal $I = \bigoplus_{n\in \N}\F_p$ is proper in $A$, so define $M$ to be a maximal ideal of $A$ containing $I$.  Note that $I$ is dense in the topology of $A$, and hence so is $M$.  In particular $M$, being proper and dense, cannot be closed in $A$.  To show that $M$ has codimension $1$, consider an element $a = (a_n)_{n\in \N}$ of $A$.  Then $a^p = ({a_n}^p)_{n\in \N} = (a_n)_{n\in \N} = a$.   Thus $A/M$ is a field in which every element is a root of $x^p - x$ and hence $A/M\iso \F_p$.  We mention in passing that replacing $\F_p$ with an arbitrary field $k$, the structure of $A/M$ can be much more complicated!
% https://math.stackexchange.com/questions/87981/infinite-product-of-fields
% Excerpt from Mumford's Red Book (p.74): "Example G. Spec(∏∞i=1k), k a field. Those familiar with ultrafilters and similar far-out mysteries will have no trouble proving that this topological space is the Stone-Cech compactification of Z+. Logicians assure us that we can prove more theorems if we use these outrageous spaces."
\end{ex}

\medskip

Note that pseudocompact algebras and modules are in particular pseudocompact $k$-vector spaces.  Pseudocompact objects need not of course be compact, since even a finite dimensional $k$-vector space is not compact when $k$ is infinite.  But a very useful weaker version of compactness holds.  Recall that an \emph{affine subspace} of a vector space $U$ is a coset of a subspace of $U$.

\begin{lemma}\label{Lemma PC is LC}
Let $U$ be a pseudocompact vector space.  Then $U$ is \emph{linearly compact}, meaning that if $\mathcal{X}$ is a collection of closed affine subspaces of $U$ and if $\bigcap_{X\in \mathcal{X}}X = \varnothing$, then there is a finite subcollection $X_1,\hdots,X_n\in \mathcal{X}$ such that $X_1\cap \hdots\cap X_n = \varnothing$.
\end{lemma}

\begin{proof}
It is more-or-less obvious that a finite dimensional $k$-vector space is linearly compact, so this result is \cite[Proposition 4]{Zelinsky}.
\end{proof}

Linearly compact algebras and modules are well-studied and there is a huge literature on the subject, see for instance \cite{Lefschetz,Warner,Zelinsky} for a general introduction.  Certain results we will discuss for pseudocompact objects hold for this larger and more difficult class of objects, but our focus is clarity of exposition rather than generality, and so we will not attempt to present results in the maximum possible generality.  

We give some simple examples of useful properties of pseudocompact algebras and modules that follow essentially from Lemma \ref{Lemma PC is LC}.  Firstly, a pseudocompact vector space $V$ has the discrete topology if, and only if, it has finite dimension: if $V$ is finite dimensional then there is an open subspace $W$ of smallest possible dimension, and this $W$ must contain every open subspace of $V$.  Now since $V$ is Hausdorff and the open subspaces of $V$ form a basis of neighbourhoods of $0$, it follows that $W$ must be $\{0\}$ and hence $V$ is discrete.  If $V$ is discrete and has infinite dimension, then choose a linearly independent set $\{v_1,v_2,v_3,\hdots\}$ of vectors in $V$.  For each $n\in \N$, define the affine subspace
$$W_n = (v_1+\hdots+v_n) + \langle v_{n+1},v_{n+2}, \hdots\rangle.$$
The intersection of a finite number of $W_n$ is non-empty, but the intersection of all the $W_n$ is empty, showing that $V$ is not linearly compact and hence not pseudocompact.  It follows from this observation that a closed subspace $W$ of a pseudocompact vector space $V$ is open if, and only if, it has cofinite dimension: one simply looks at the inverse image of $0$ under the continuous projection $V\to V/W$.

\begin{lemma} [adaptation of Proposition 0.3.3 from \cite{wilson}] \label{lemmaWilson}
Let $A$ be a pseudocompact vector space and let $\mathcal{I}$ be a collection of open subspaces of $A$ such that whenever $X,Y\in \mathcal{I}$, there is $Z\in \mathcal{I}$ contained in $X\cap Y$, and such that $\bigcap_{X\in \mathcal{I}} X=0$.  If $C$ is a closed subspace of $A$, then
$$
	C=\bigcap_{X\in \mathcal I} (C+X).
$$
\end{lemma}

\begin{proof}
If $a\notin C$ then 
$$(a+C)\cap \bigcap_{X\in \mathcal I} X=\varnothing.$$ By Lemma \ref{Lemma PC is LC} there are $X_1,\dots,X_n\in \mathcal I$ so that 
$$(a+C)\cap (X_1\cap \dots \cap X_n)=\varnothing.$$
There is $Y\in \mathcal I$ such that $Y\subseteq X_1\cap\dots\cap X_n$. So $(a+C)\cap Y=\varnothing$ , hence $a\notin C+Y$.  We thus obtain the inclusion
$$
	\bigcap_{X\in \mathcal I} (C+X) \subseteq C+\bigcap_{X\in \mathcal I} X = C,
$$
the other being trivial.
\end{proof}

\begin{corol}\label{Corol closed ideal in max ideal}
Every proper closed (left) ideal of $A$ is contained in a maximal closed (left) ideal of $A$. 
\end{corol}

\begin{proof}
Let $\mathcal{I}$ denote the set of open ideals of $A$.  If $C$ is a proper closed (left) ideal of $A$ then, since $C = \bigcap_{X\in \mathcal I} (C+X)$  by Lemma \ref{lemmaWilson}, some $C+X$ is proper in $A$. Hence $C$ is contained in a proper open left ideal $C+X$ (which is an ideal if $C$ is) and the result follows because $A/(C+X)$ is finite dimensional.
\end{proof}

%[[Old paragraph, whose content is in example 2.6]] Recall that a \textit{simple} pseudocompact $A$-module $S$ is a non-zero  pseudocompact module which does not have non-zero proper closed submodules.  A pseudocompact module is simple if, and only if, it is isomorphic to a module of the form $A/M$ with $M$ a maximal closed left ideal in $A$. As any maximal closed left ideal is open (by Lemma \ref{lemmaWilson}), therefore any simple pseudocompact $A$-module is finite-dimensional. 

We give a (probably well-known) technical lemma.  Recall that two proper left ideals of a ring $A$ are said to be coprime when their sum is $A$.

\begin{lemma}
Let $X, Y, Z$ be open left ideals of $A$ with $Z$ maximal.  If neither $A/X$ nor $A/Y$ have composition factors isomorphic to $A/Z$, then $X\cap Y$ and $Z$ are coprime.  
\end{lemma}

\begin{proof}
As $Z$ is maximal, we need only check that $X\cap Y\not\subseteq Z$.  If it were, then we would have a projection $A/(X\cap Y)\to A/Z$ and hence $A/Z$ would appear as a composition factor of $A/(X\cap Y)$.  But the composition factors of $A/(X\cap Y)$ are the factors of $A/X$ and the factors of $X/(X\cap Y) \iso (X+Y)/Y$, so $A/Z$ does not appear.
\end{proof}

\begin{lemma}\label{corol prod of simples is cyclic}
Let $A$ be a pseudocompact algebra and $\mathcal{X}$ the set containing one  representative of each isomorphism class of simple pseudocompact left $A$-module and no other elements.  The module $M = \prod_{X\in \mathcal{X}}X$ is cyclic.
\end{lemma}

\begin{proof}
Take as our representative of each isomorphism class the module $A/I$ with $I$ some maximal closed left ideal of $A$.  We claim there is a surjective module map $A\to M$.  As $\invlim$ is exact, it is sufficient to check that the natural map
$$A\to \prod_{i=1}^n A/I_i$$
is surjective for each finite set of $A/I_1,\hdots,A/I_n\in \mathcal{X}$, and for this we follow the proof of the Chinese Remainder Theorem.  If $n=1$, surjectivity is immediate, so suppose the result holds given fewer than $n$ left ideals.  The left ideals $I_1\cap \hdots\cap I_{n-1}$ and $I_n$ are coprime by the previous lemma and so we can write $1 = x+y$ with $x\in I_1\cap \hdots \cap I_{n-1}$ and $y\in I_n$.  Given $(b+(I_1\cap\hdots\cap I_{n-1}), c + I_n)$, the element $a = by+cx$ maps onto it, so the map $A\to A/(I_1\cap\hdots \cap I_{n-1})\oplus A/I_n$ is onto.  But by induction, the map 
$$A/(I_1\cap\hdots \cap I_{n-1})\to A/I_1\oplus \hdots\oplus A/I_{n-1}$$
is onto, and hence the map $A\to A/I_1\oplus \hdots\oplus A/I_{n}$ is onto, as required.
\end{proof}

\begin{lemma}\label{lemma inclusion in a semisimple module splits}
Let $A$ be a pseudocompact algebra, $Y$ an indexing set and $V = \prod_{y\in Y}S_y$ a product of simple pseudocompact $A$-modules.  If $U$ is a closed submodule of $V$ then $U$ is a continuous direct summand of $V$.
%Let $U$ be a closed submodule of $V$.  Then there is a subset $X$ of $Y$ such that
%$$V = U \oplus \prod_{y\in X}S_y.$$
\end{lemma}

\begin{proof}
This proof is dual to the abstract case (see for instance, \cite[Prop. 1.7.1]{StBo}).  Denote by $\mathcal{X}$ the set of subsets $X$ of $I$ such that the composition 
$$U\xrightarrow{\iota}V\xrightarrow{\pi_X}\prod_{i\in X}S_i$$
is surjective, ordered by inclusion.  Then $\mathcal{X}\neq\varnothing$ because $\varnothing\in \mathcal{X}$.  Let $\mathcal{C}$ be a chain in $\mathcal{X}$. Using the notation from Example \ref{Ex.products}, we have natural projection maps $\varphi_{X} = \varphi_{XY}$ for each $X\in \mathcal{C}$ and these maps induce a so-called map of inverse systems (cf.\ \cite[\S 1.1]{RibZal}) $\{\varphi_{X}\iota : U \to \prod_{i\in X}S_i\}$.  
%Considering inverse systems $\{U,\textrm{id}\}$ and $\{\prod_{i\in X}S_i,\pi_{X,X'}\}$ we have that each 
%map in the set
%$\{\pi_X\iota:U\to \prod_{i\in X}S_i \,|\,X\in \mathcal{C}\}$ 
Each of these maps is onto and hence the limit map
$$
    \invlim \pi_X\iota:  U \to \invlim_{{}_{X\in \mathcal{C}}} \prod_{i\in X}S_i = \prod_{i\in \bigcup X}S_i
$$
is onto by the exactness of $\invlim$. 
%As (see Example \ref{Ex.products})
%$$
%    \invlim U=U,\quad \invlim\prod_{i\in X}S_i\,= \prod_{i\in \bigcup X}S_i
%$$ 
It follows that $\bigcup X$ is an upper bound for $\mathcal C$ in $\mathcal X$. So by Zorn's Lemma, $\mathcal{X}$ contains a maximal element $X$.

%The set of maps $\{\pi_X\iota\,|\,X\in \mathcal{C}\}$ forms a surjective map of inverse systems (check!) and the limit map is $\pi_{\bigcup X}\iota$, which is surjective by linear compactness (check all this!).  

We claim that the composition $\pi_X\iota$ is injective (and hence an isomorphism).  If not, then $U\cap \prod_{i\in I\backslash X}S_i\neq 0$ and hence there is an element $u = (u_i)$ of $U$ such that $u_i=0$ for every $i\in X$ but $u_j\neq 0$ for some $j\not\in X$.  Abusively denoting by $S_j$ and $\prod_{i\in X}S_i$ the images of $S_j$ and $\prod_{i\in X}S_i$ under $\pi_{X\cup \{j\}}$, we note that $S_j\cap \pi_{X\cup\{j\}}(U)\neq 0$ and hence, $S_j$ being simple, $S_j\subseteq \pi_{X\cup\{j\}}(U)$.  Given $z\in \prod_{i\in X}S_i$, we know since $\pi_X\iota$ is onto that there is an element $a\in U$ such that $\pi_{X\cup\{j\}}(a) = z + z'$ with $z'\in S_j$.  But $z' = \pi_{X\cup \{j\}}(b)$ for some $b$ and hence $z = \pi_{X\cup\{j\}}(a-b)$.  Thus $S_j$ and $\prod_{i\in X}S_i$ are contained in $\tn{Im}(\pi_{X\cup\{j\}}\iota)$, showing that $\pi_{X\cup\{j\}}\iota$ is surjective and contradicting the maximality of $X$.

It follows that $\pi_X\iota$ is an isomorphism, so that $\iota$ splits via $(\pi_X\iota)^{-1}\pi_X$.
\end{proof}

\iffalse
\begin{proof}
This proof is very like the proof for finite dimensional modules [REF], plus an application of Zorn's Lemma (cf [I.7.1 of that book]).  Denote by $\mathcal{M}$ the set of subsets $Z$ of $Y$ such that $U\cap \prod_{y\in Z}S_y = 0$, ordered by inclusion  Certainly $\mathcal{M}$ is non-empty, since it contain the empty set.  
\end{proof}
\fi

The Krull-Schmidt Theorem \cite[Corollary 19.22]{Lam91} tells us that a finitely generated module for a finite dimensional algebra can be decomposed as a direct sum of finitely many indecomposable modules in an essentially unique way.  We have just seen in Lemma \ref{corol prod of simples is cyclic} that a finitely generated module for a pseudocompact algebra may well be an infinite product of indecomposable modules.  The following very useful analogue of the Krull-Schmidt Theorem tells us that this is essentially as bad as it can get, and that decompositions of a finitely generated module as a product of indecomposables are essentially unique, in a precise way:

\begin{prop}(\cite[Proposition 3.5 and before]{MSZ})\label{prop exchange property}
Let $A$ be a pseudocompact algebra and $M$ a finitely generated pseudocompact $A$-module. 
\begin{enumerate}
    \item $M$ is isomorphic to a direct product of indecomposable $A$-modules.
    \item (the Exchange Property) If $S$ is a set and if $X$ is a continuous direct summand of the module $Y = \prod_{s\in S}Y_s$, where each $Y_s$ is an indecomposable direct summand of $M$, then there exists a subset $T$ of $S$ with the property that
    $$Y = X \oplus \prod_{s\in T}Y_s.$$
\end{enumerate}
\end{prop}

Given closed ideals $I, I'$ of the pseudocompact algebra $A$, denote by $I\cdot I'$ the topological closure of the abstract product of $I$ and $I'$.  Given a closed ideal $I$ and $n\geqslant 2$, define recursively the closed ideal $I^n$ to be  $I^{n-1}\cdot I$.

\begin{ex}
It is intuitively clear that the abstract product of two ideals need not be closed, but we did not encounter explicit examples in the literature so we present one here.  Let $A = k\db{x_1, x_2, \hdots}$ be the pseudocompact algebra of power series in infinitely many variables $x_i$. 
%There is an obvious continuous surjective algebra homomorphism $\varphi_n : A\to A_n = k\db{x_1, \hdots, x_n}$ for each $n\in \N$ sending $x_i$ to itself when $i\leqslant n$ and to $0$ otherwise. %, and indeed $A = \invlim A_n$.  
Let $I$ be the closed ideal of $A$ generated by the $x_i$.  The element $x = \sum_{i\in \N}{x_i}^2$ is in $I^2$, but it cannot be written as a finite sum of monomials, and hence is not an element of the abstract product of $I$ with itself.

Note that $I = J(A)$ (cf.\ \S\ref{subsection Jacobson}), so this is an example of a pseudocompact algebra $A$ for which $J^2(A)$ is not the abstract product of $J(A)$ with itself.
\end{ex}

\section{The Jacobson radical and semisimple pseudocompact algebras}\label{section Jacobson radical and ss}

The main result of this section is a long and familiar (to those who work with finite dimensional algebras) characterization of semisimple pseudocompact algebras. We extend the fundamental Artin-Wedderburn theorem, showing that a semisimple pseudocompact algebra is a product of matrix algebras over finite dimensional division algebras. As with Artinian algebras, the Jacobson radical is a useful tool to measure the failure of semisimplicity of a pseudocompact algebra. So we start by defining a topological version of the Jacobson radical and compare the definition with the classical one.  For finite dimensional algebras there are a huge number of equivalent descriptions of the Jacobson radical, which may not be equivalent for an arbitrary algebra.  We observe that (following a common pattern) the situation for pseudocompact algebras is as equally well behaved as the situation for finite dimensional algebras.

\subsection{The Jacobson radical}\label{subsection Jacobson}

Let $A$ be a pseudocompact algebra.  The \textit{topological Jacobson radical} $J(A)$ of $A$ is the intersection of the maximal closed  left ideals of $A$ 
(which by Lemma \ref{lemmaWilson} is the intersection of the maximal open left ideals of $A$).
%(note that maximal closed left ideals are necessarily open).  
% By $J(A)$ we denote its \textit{topological Jacobson radical}, that is the intersection of the maximal closed  left ideals of $A$ 
% (which by Lemma \ref{lemmaWilson} is the intersection of the maximal open left ideals of $A$).
%(note that maximal closed left ideals are necessarily open).  
%We say that a closed ideal $I$ of $A$ is \emph{pronilpotent} if $\bigcap_{n\in \N}I^n = 0$ [REF?].  

\begin{ex}\label{examples Jacobson radicals}
\begin{enumerate}
\item Suppose that $A=k\db{x}$. Then $A$ has a unique maximal closed left ideal, generated by $x$, hence $J(A)=(x)$ has codimension 1 in $A$. 
%Moreover, it is easy to see that $J(A)$ is a pronilpotent ideal.  
Compare this with the polynomial algebra $k[x]$, whose Jacobson radical is $0$.

\item Consider the pseudocompact algebra 
$$A = \prod_{\N}k = k_0\times k_1 \times\hdots,$$ 
the product of a countable number of copies of the algebra $k$.  For each $n\in \N$, the closed ideal 
$$I_n = k_0\times\hdots\times k_{n-1}\times 0 \times k_{n+1}\times \hdots$$
has codimension 1 and so is maximal.  As $\bigcap_{n\in \N}I_n = 0$, it follows that $J(A) = 0$.

\item\label{Examples JacobsonRadical profinite groups} If $G$ is a profinite group, and $\cchar k$ does not divide the order of any continuous finite quotient group of $G$, then $J(k\db{G})=0$ -- this follows using Lemma \ref{radical limit of radicals} below and Maschke's Theorem for finite groups (see for instance \cite[Theorem 15.6]{CR62}).

\item Generalizing the first example, let $Q$ be any quiver.  Then $J(k\db{Q})$ is the closed ideal generated by the arrows of $Q$. This fact follows by \cite[Lemma 2.12]{IM17}, whose proof, given the results of the next section, does not require the corresponding semisimple algebra to be finite dimensional. %Alternatively, it follows by duality from [REF     for colagebras!].  
Thus the completed path algebra gives a much better tool for the study of finite dimensional algebras treated as quotients of path algebras than the abstract path algebra, since for a finite dimensional algebra of the form $A = kQ/I$ with $I$ an admissible ideal, the Jacobson radical is generated as an ideal by the arrows of $Q$.  Meanwhile for even a finite quiver $Q$ with loops or cycles, the radical of the abstract path algebra $kQ$ is awkward to describe: it is the ideal with $k$-basis the so-called ``regular paths'' in $Q$ (we could not find a reference for precisely this fact but it follows from the more general \cite[Proposition 1.3]{CoelhoLiu}).

\end{enumerate}
\end{ex}

 %In particular it implies (equivalence \ref{P1.IntLeft} $\Longleftrightarrow$  \ref{P1.IntAll}) that topological Jacobson radical coincides with the classical. 
The following proposition gives a characterization of $J(A)$.  We say that a closed ideal $I$ of $A$ is \emph{pronilpotent} if $\bigcap_{n\in \N}I^n = 0$.  An element $a$ of an abstract algebra $A$ is a \emph{(topological) non-generator} if whenever $X$ is a subset of $A$ and the (closed) left ideal generated by $X\cup \{a\}$ is $A$, then the (closed) left ideal generated by $X$ is already $A$.

\begin{prop} \label{radicalChar}
Let $A$ be a pseudocompact algebra. The following subsets of $A$ are equal: 
\begin{enumerate}[label=\textnormal{(\arabic*)}]
\item The intersection of the maximal closed left ideals of $A$; \label{P1.IntLeft}
\item The intersection of the maximal closed right ideals of $A$;
\label{P1.IntRight}
\item The intersection of the maximal closed two sided ideals of $A$;
\label{P1.IntBilat}
\item The intersection of the (not necessarily closed) maximal left ideals of $A$; 
\label{P1.IntAll}
\item $\{x\in A\ |\  1-yx \mbox{ is left invertible for any} \ y\in A  \}$;
\label{P1.SetLeftInv}
\item $\{x\in A\ |\  1-yxz \mbox{ is invertible for any} \ y,z\in A  \}$; 
\label{P1.SetInv}
\item The set of all non-generators of $A$; \label{P1.NoGen}
\item The set of all topological non-generators of $A$;\label{P1.TopNoGen}
\item The intersection of the annihilators of the simple pseudocompact left $A$-modules;\label{P1.IntAnn}
\item The smallest closed submodule $V$ of $A$ such that $A/V$ is a product of simple $A$-modules;
\label{P1.ProdSimples}
%\item Intersection of maximal closed  ideals;
\item The maximal closed pronilpotent ideal.
\label{P1.Pronilp}
\end{enumerate}
\end{prop}

\begin{proof}

[$\ref{P1.IntAll}=\ref{P1.SetLeftInv}=\ref{P1.SetInv}=\ref{P1.NoGen}$] is classical  (for instance, see \cite[Lemma 4.1, Lemma 4.3]{Lam91} for [$\ref{P1.IntAll}=\ref{P1.SetLeftInv}=\ref{P1.SetInv}$] and \cite[Page 65]{BenKei} for [$\ref{P1.IntAll}=\ref{P1.NoGen}$]).% \cite{CR06,ARS}].

[$\ref{P1.IntLeft} = \ref{P1.IntAnn}$] is also completely standard, given that $A/I$ is a simple pseudocompact $A$-module whenever $I$ is a maximal closed ideal of $A$ and every pseudocompact simple module is isomorphic to one of this form.

The proof that [$\ref{P1.IntLeft} = \ref{P1.TopNoGen}$] is just like the proof that [$\ref{P1.IntAll} = \ref{P1.NoGen}$]: if $a$ is a topological non-generator and $M$ is a maximal closed left ideal, then $a\in M$ because otherwise $\{a\}\cup M$ generates $A$ while $M$ doesn't.  For the reverse inclusion, if the left ideal generated by the set $X$ is not $A$ then it's contained in a maximal closed left ideal $M$ by Corollary \ref{Corol closed ideal in max ideal}.  If $a$ is in $\ref{P1.IntLeft}$ then $a\in M$ and hence $\{a\}\cup X$ also doesn't generate $A$.

[$\ref{P1.IntLeft} = \ref{P1.IntBilat}$] follows by the equality of these sets for finite dimensional algebras.  Denote by $J(B)$ the intersection of the maximal closed left ideals of $B$ and by $r(B)$ the intersection of the maximal closed two sided ideals of $B$.  Given an open two sided ideal of $A$, $J(A/I) = r(A/I)$ by the result for finite dimensional algebras \cite[Corollary 3.1.8.]{DK1994}.  By the correspondence theorem, it follows that the intersection of the maximal closed left ideals containing $I$ is equal to the intersection of the maximal closed two sided ideals containing $I$.  Since $\bigcap_{I\lhd_O A}I = 0$, a standard compactness argument shows that every maximal open left ideal $L$ contains an open two sided ideal: We have $L^C$ is closed and $L^C\cap \bigcap_{I\lhd_O A}I = \varnothing$.  Hence by Lemma \ref{Lemma PC is LC} there is a finite set of open ideals $I_1,\hdots,I_n$ such that
$$L^C\cap I_1\cap \hdots\cap I_n = \varnothing$$
and the open ideal $I_1\cap \hdots \cap I_n$ is contained in $L$.  Thus
$$J(A) = \bigcap_{I\lhd_O A}\bigcap_{\hbox{max left }M\supseteq I}M
= \bigcap_{I\lhd_O A}\bigcap_{\hbox{max two-sided }M\supseteq I}M
= r(A).$$

[$\ref{P1.IntLeft}\subseteq \ref{P1.SetLeftInv}$] Let $x\in \ref{P1.IntLeft}$ and $y\in A$ be such that $1-yx$ is not left invertible. Thus $A(1-yx)\subsetneq A$. But $A(1-yx)$ is closed, being finitely generated as a left $A$-module \cite[Lemma 2.1]{Brumer} and so it is contained in some maximal closed left ideal $M$ of $A$ by Corollary \ref{Corol closed ideal in max ideal}.  Hence $1\in yx+M$, but $yx+M=M$, a contradiction. 

The inclusion [$\ref{P1.IntRight}\subseteq \ref{P1.SetLeftInv}$] is similar because $\ref{P1.SetInv}$ is left-right symmetric, so we can swap between left and right as we choose. 

The inclusion [$\ref{P1.IntAll}\subseteq \ref{P1.IntLeft}$] is immediate: a maximal closed left ideal is maximal as an abstract left ideal because an ideal containing an open ideal is open. By left-right symmetry, we obtain $[\ref{P1.IntAll}\subseteq \ref{P1.IntRight}]$ in the same way.  

%The inclusions [$\ref{P1.SetInv}\subseteq \ref{P1.IntLeft}$] and 
%[$\ref{P1.SetInv}\subseteq \ref{P1.IntRight}$] are easy, since for instance $\ref{P1.SetInv} = \ref{P1.IntAll}\subseteq \ref{P1.IntLeft}$.

%follows easily. Indeed, $\ref{P1.SetInv}$ equals the intersection of all left ideals in $A$ which is a subset of $\ref{P1.IntLeft}$.  

[$\ref{P1.IntAnn} \subseteq \ref{P1.ProdSimples}$]  If $x\in \ref{P1.IntAnn}$ and $V$ is a closed submodule with $A/V$ a product of simple modules, then $x$ is in the kernel of the natural projection $A\to A/V$, so $x\in V$.

[$\ref{P1.ProdSimples} \subseteq \ref{P1.IntLeft}$] Let $J$ denote the intersection of the maximal closed left ideals of $A$ and denote by $\mathcal{M}$ the set of all maximal closed left ideals.  The natural projections $A/J \to A/M$ ($M\in \mathcal{M}$) yield a continuous map
$$\gamma:A/J \to \prod_{M\in \mathcal{M}}A/M,$$
whose kernel is $0$.  By Lemma \ref{lemma inclusion in a semisimple module splits}, this inclusion splits.  The product $X$ of one copy of each isomorphism class of simple $A$-module is finitely generated by Lemma \ref{corol prod of simples is cyclic}, so by Lemma \ref{prop exchange property},
$$\prod_{M\in \mathcal{M}}A/M = \gamma(A/J) \oplus \prod_{M\in \mathcal{N}}A/M$$ 
for some subset $\mathcal{N}$ of $\mathcal{M}$.  Thus $A/J$ is isomorphic to a product of simple modules, and hence $V\subseteq J$. %[note that $V$ in fact exists, being the intersection of the set $\mathcal{X}$ of closed submodules $X$ such that $A/X$ is a product of simples.  The injection $A/\cap X \to \prod A/X$ is injective and is isomorphic to some product of simples by the same argument as above, I think]

%[$\ref{P1.IntLeft} \subseteq \ref{P1.ProdSimples}$]  If $S$ is a simple module and $x\in J$ then $x$ is in the kernel of $A\to S$ because the kernel is a maximal left ideal.  Hence $x$ is in the kernel of the natural projection $A\to A/V$, and hence $x\in V$.  So $J\subseteq V$ as required.

[$\ref{P1.IntLeft}\subseteq \ref{P1.Pronilp}$]. Let $J(A)$ denote the intersection of the maximal closed left ideals of $A$ and let $I$ be an open left ideal.  Then $J(A)+I\subseteq J(A/I)$ and hence $J(A)^n\subseteq I$ for some $n$ by the nilpotence of the radical for finite dimensional algebras (see for instance \cite[Proposition 3.1.9]{DK1994}).  Thus $\bigcap_{n\in \N}J(A)^n\subseteq \bigcap_{I\lhd_O A} I = 0$.
%Then $\bigcap J^n=0$ for any finite dimensional quotient of $A$, Hence  $\bigcap J^n=0$, and $J$ is pronilpotent. 

[$\ref{P1.Pronilp}\subseteq \ref{P1.IntLeft}$]. Let $M$ be a closed maximal left ideal and $P$ a pronilpotent ideal in $A$. The module $A/M$ is simple, hence $P(A/M)$ equals $0$ because $P$ is pronilpotent. Therefore $P=PA\subseteq M$, so $P\subseteq \ref{P1.IntLeft}$. 
\end{proof}

Note that the equality of $\ref{P1.IntLeft}$ and $\ref{P1.IntAll}$ in the above proposition says that the topological Jacobson radical of the pseudocompact algebra is equal to the Jacobson radical of $A$ treated as an abstract algebra.  As there is no ambiguity, we will from now on refer to $J(A)$ simply as the Jacobson radical of $A$.  Item \ref{P1.IntBilat} is the definition of the Jacobson radical given in \cite[Page 444]{Brumer}.

The Jacobson radical $J(A)$ is well-behaved with respect to radicals of finite dimensional quotients of $A$, in the sense that $J(A)$ coincides with the inverse limits of $J(A/I)$ where $I$ runs over all open ideals in $A$. The latter implies that surjective homomorphisms between pseudocompact algebras map  radicals onto radicals. These two properties were proved, for instance, in \cite{IM17}. For the sake of completeness we state them here with complete proofs.

\begin{lemma}\label{radical limit of radicals}
Write $A = \invlim_{I} \{A/I\,,\,\alpha_{II'}\}$ as an inverse limit of finite dimensional quotient algebras $A/I$.  Then
$$J(A) = \invlim_I J(A/I).$$
\end{lemma}

\begin{proof}
Since the quotients $\alpha_{II'}:A/I'\to A/I$ are surjective, it is immediate that $\alpha_{II'}(J(A/I'))\subseteq J(A/I)$.  Thus, the restriction of the inverse system of $A/I$ to their Jacobson radicals indeed yields an inverse system with inverse limit $\invlim J(A/I)\subseteq A$.  Since an element $x\in J(A)$ maps into each $J(A/I)$, it follows that $J(A)\subseteq \invlim J(A/I)$.  On the other hand, given $x\not\in J(A)$, there is some open maximal left ideal $M$ not containing $x$.  Working within the cofinal subsystem of $A/I$ with $I\subseteq M$, we see that $x+I\not\in J(A/I)$, and hence $x\not\in\invlim J(A/I)$, so that $\invlim J(A/I)\subseteq J(A)$.
\end{proof}

\begin{corol}\label{radical surjective}
Let $A,B$ be pseudocompact algebras and let $\alpha:A\to B$ be a continuous surjective algebra homomorphism.  Then $\alpha(J(A)) = J(B)$. 
\end{corol}

\begin{proof}
That $\alpha(J(A))\subseteq J(B)$ is immediate.  To prove the other inclusion, we begin by supposing that $A, B$ are finite dimensional.  Recall \cite[Proposition 3.5]{ARS} that the radical $\tn{Rad}_B(U)$ of a finitely generated $B$-module $U$ is given by $J(B)U$ (and similarly for $A$).  We can treat $B$ either as a $B$-module, or as an $A$-module via $\alpha$, and since $\alpha$ is surjective it follows that $\tn{Rad}_BB = \tn{Rad}_AB$.  Thus
$$
	J(B) = J(B)B = \tn{Rad}_B B = \tn{Rad}_A B = J(A)\cdot B = \alpha(J(A)) B \subseteq \alpha(J(A)). 
$$

Now let $A$ be general and $B$ finite dimensional.  Since $B$ is discrete and $\alpha$ is continuous, we can consider a cofinal subset of open ideals of $A$ contained in the kernel of $\alpha$.  By factorizing $\alpha$ through these quotients we obtain a map of inverse systems $\{\alpha_I:A/I\to B\,|\,I\lhd_O A, I\leqslant \tn{Ker}(\alpha)\}$.  Restricting and corestricting this inverse system to the Jacobson radicals, we obtain a surjective map of inverse systems $\alpha_I:J(A/I)\onto J(B)$, whose inverse limit is $\alpha:J(A)\to J(B)$ by Lemma \ref{radical limit of radicals}.  It is onto by the exactness of $\invlim$.

Finally, allowing both $A$ and $B = \invlim\{B/K,\beta_{KK'}\}$ to be general, the obvious composition $\beta_K\alpha:A\to B\to B/K$ is a surjective map onto the finite dimensional algebra $B/K$ and hence it restricts to a surjection $J(A)\onto J(B/K)$ for each $K$.  We obtain in this way a surjective map of inverse systems and the result follows from the exactness of $\invlim$.
\end{proof}

% As the set $\ref{P1.IntAll}$ equals $\ref{P1.IntLeft}$, the topological Jacobson radical equals usual Jacobson radical for the class of pseudocompact algebras. On the other hand, as $\ref{P1.IntLeft}$ equals $\ref{P1.IntBilat}$, the caracterization above coincides with Jacobson radical of pseudocompact algebras defined in \cite{Brumer} [[Rewrite!]].

\subsection{Topologically semisimple algebras} 

Recall that a (classically) semisimple algebra is an associative algebra for which every left ideal has a complement.
%Artinian algebra with zero Jacobson radical.  
By the Wedderburn--Artin theorem \cite[Theorem 3.5]{Lam91}, a semisimple $k$-algebra is a direct product of a finite number of matrix algebras over $k$-division algebras.% In particular, a classically semisimple algebra is necessarily Artinian.
%For instance, every group algebra of a finite group over a field of characteristic coprime with the order of the group is a semisimple algebra.  
%We extend in a natural way the notion of semisimplicity to topological algebras, and we will see in this subsection that a direct analogue of the  Wedderburn--Artin Theorem holds for pseudocompact algebras.

\begin{defn}
Let $A$ be a topological algebra. We say that $A$ is \textit{(left) topologically semisimple} if for any closed left ideal $I$ in $A$ there exists a closed left ideal $L$ such that $I\oplus L=A$. 
\end{defn}

\begin{ex}
Let $X$ be an infinite set and consider the algebra $A=\prod_{i\in X} k$. Then $A$ is not (classically) semisimple: the (non-closed) ideal $I=\bigoplus_{i\in X} k$ has no complement, because every non-zero ideal of $A$ intersects $I$. But it follows from Proposition \ref{ssCharacterization} that $A$ \emph{is} topologically semisimple.
\end{ex}

The following is a Wedderburn-Artin Theorem for pseudocompact algebras.  Several of the equivalences here are presented for a wider class of algebras in \cite[Theorem 3.10]{IZR06}.

\begin{prop} \label{ssCharacterization}
Let $A$ be a pseudocompact algebra. The following conditions are equivalent.
\begin{enumerate}[label=\textnormal{(\arabic*)}]
\item $A$ is (left) topologically semisimple; \label{P2.topSS}
\item $J(A)=0$; \label{P2.J=0}
\item The free left $A$-module $A$ is isomorphic to a product of simple modules; \label{P2.prodSimples}
\item Every left pseudocompact $A$-module is projective; \label{P2.everyProj}
\item Every left pseudocompact $A$-module is injective; \label{P2.everyInj}
\item $A$ is an inverse limit of finite-dimensional semisimple algebras. \label{P2.invLimSS}
\item $A$ is isomorphic to a product of full matrix algebras $M_{n_i}(D_i)$, where the $n_i$ are positive integers and the $D_i$ are finite dimensional $k$-division algebras.  This product is unique up to permutation of the terms and isomorphism of the division algebras $D_i$.
\label{P2.prodFullM}
\end{enumerate}
\end{prop}

\begin{proof}

[\ref{P2.topSS}$\Rightarrow$\ref{P2.J=0}] As $A$ is topologically semisimple and $J(A)$ is closed, $A=J(A)\oplus I$, for some closed left ideal $I$.  If $J(A)\neq 0$ then $I\neq A$ so $I$ is contained in a closed maximal left ideal $M$.  But $M$ does not contain $J(A)$, since otherwise $A = J(A)+I\subseteq M$, a contradiction. 

[\ref{P2.J=0}$\Rightarrow$\ref{P2.prodSimples}] This follows by Proposition \ref{radicalChar} because $A = A/J(A)$ is a product of simple $A$-modules. 

[\ref{P2.prodSimples}$\Rightarrow$\ref{P2.everyProj}]
Any left pseudocompact $A$-module $U$ is an inverse limit of discrete finite dimensional $A$ modules. An inverse limit of projective $A$-modules is a projective $A$-module by \cite[Corollary 3.3]{Brumer}, so it is enough to show that any discrete finite dimensional $A$-module is projective.  Given such a $U$, let $\gamma: A^n\onto U$ be a continuous surjection.  The kernel of $\gamma$ is open in $A^n = \prod_{i\in X}S_i$ ($S_i$ simple), and so $\gamma$ factors through some product $T$ of finitely many $S_i$.  The map $A^n\to T$ splits and the map $T\to U$ splits by the finite dimensional version of this result.  Hence $\gamma$ splits, showing that $U$ is projective.

% we present it as a quotient of $A^n$ by some open ideal $I$. As $A^n$ is a product of simple modules $\prod_{i\in \mathcal X} S_i$, and $I$ (being open in $A$) is a product  
% $\prod_{i\in \mathcal X} T_i$ where $T_i=S_i$ for all but finitely many $i\in \mathcal X$, we have that $U$ is a quotient of module which is direct sum of simples $A$-modules, hence $U$ is projective.

%If $A$ is a product of simple modules then any  any short exact sequence with a middle term a free $A$-module splits (?), hence any module is direct summand of a free module and is projective. [I don't understand this]

[\ref{P2.everyProj}$\Rightarrow$\ref{P2.everyInj}] Given a module $U$ and an injective map of modules $\gamma:U\to V$, the short exact sequence $U\to V \to V/\gamma(U)$ splits since the third module is projective.  So $U$ is injective.

[\ref{P2.everyInj}$\Rightarrow$\ref{P2.topSS}] Suppose that $I$ is a closed left ideal in $A$. The short exact sequence $I\to A \to A/I$ splits because $I$ is injective. Hence $A\iso I\oplus A/I$ and $A$ is (left) topologically semisimple.

%[\ref{P2.invLimSS}$\Rightarrow$\ref{P2.J=0}]. Suppose that $A=\invlim\{ A_{ij}, \phi_{ij} \}$, with $A_{ij}$ are semisimple. Then by Lemma (?) $J(A)=\invlim J(A_{ij})=0$. [unnecessary!]

[\ref{P2.J=0}$\Rightarrow$\ref{P2.invLimSS}]. Suppose that $J(A)=0$. 
Write $A=\invlim \{A_i, \varphi_{ij}\}$ with each $A_i$ a finite-dimensional algebra and each $\varphi_{ij}$ a surjective algebra map.  By Corollary \ref{radical surjective} we have $J(A_i)=\varphi_i(J(A))=0$, hence each $A_i$ is semisimple.

[\ref{P2.invLimSS}$\Rightarrow$\ref{P2.prodFullM}] Denote by $C$ the complete set of centrally primitive central idempotents of $A$.  By \cite[IV\,\S3\,Corollaries 1, 2]{Gabriel_thesis} we have $A = \prod_{e\in C} Ae$ as an algebra.  We must check that each $Ae$ is a matrix algebra.  Write $A = \invlim\{A_i, \varphi_{ij}\}$ as an inverse limit of finite dimensional semisimple quotient algebras of $A$.  
%Then $Ae = \invlim A_i\varphi_i(e)$.  
Denoting by $C_i$ the corresponding complete set for $A_i$, we have that the $C_i$ form a direct system via the maps $\gamma_{ij} : C_i\to C_j$ sending $f\in C_i$ to the unique element $c$ of $C_j$ such that $\varphi_{ij}(c)f\neq 0$.  Furthermore, $C = \dirlim \{C_i, \gamma_{ij}\}$ by \cite[Proposition 6.4]{MSbrauer}, whose proof goes through for arbitrary fields.  The maps $\gamma_{ij}$ are injective because, as the $A_i$ are semisimple, the image of an irreducible factor of $A_j$ under $\varphi_{ij}$ is either irreducible or $0$.  Hence the maps $\gamma_i : C_i \to C$ are injective.  It follows that $A_i\varphi_i(e)$ is $0$ or a matrix algebra.  Consider the cofinal inverse system of $A_i$ such that $A_i\varphi_i(e)\neq 0$.  Within this cofinal inverse system the maps $\varphi_{ij}|_{A_j\varphi_{j}(e)}$ are injective, because the kernel of such a map is a proper ideal of a matrix algebra, so must be $0$.  It follows that $\varphi_i|_{Ae}$ is injective, so that $Ae$ is isomorphic to a unital subalgebra of $M_n(D)$, where $D$ is a finite $k$-division algebra.  Hence $Ae$ has the required form.  This implication also follows from the more general \cite[Theorem 29.7]{Warner}.

%[[OBS: this implication is a special case of \cite[Theorem 3.10]{IZR06}, which itself references an old book.  How to present this? I WROTE SOMETHING HERE AND BEFORE THE RESULT]]

[\ref{P2.prodFullM}$\Rightarrow$\ref{P2.J=0}] For any $t\in I$ the closed ideal $M_t=\prod_{i\in I, i\neq t} M_{n_i}(D_i)$ is maximal in the algebra $A = \prod_{i\in I} M_{n_i}(D_i)$, so that $J(A)\subseteq \bigcap_{t\in I} M_t=0$.

%Writing $A$ as an inverse limit  of finite dimensional quotient algebras $\invlim \{A_i, \varphi_{ij}\}$ and denoting by $C_i$ the corresponding complete set for $A_i$, we have that the $C_i$ form a direct system via the maps $\gamma_{ij} : C_i\to C_j$ defined by sending $f\in C_i$ to the unique element $c$ of $C_j$ such that $\varphi_{ij}(c)f\neq 0$.  Furthermore, $C = \dirlim \{C_i, \gamma_{ij}\}$.  Supposing now that $(6)$ holds, each $A_ie_i$ is a matrix algebra.  The image of an irreducible factor of $A_j$ under $\varphi_{ij}$ is either an irreducible factor of $A_i$ or $0$.  Hence each map $\gamma_{ij}$ is injective.

\end{proof}

\begin{rem}\label{rem.locality}
By \cite[IV. §3, Corollary 3]{Gabriel_thesis}, idempotents can be lifted modulo $J(A)$ in a pseudocompact algebra $A$: that is, for any idempotent $f$ of $A/J(A)$, there is an idempotent $e\in A$ such that $e + J(A) = f$.  Since $A/J(A)$ is a product of matrix algebras by the above characterization, every non-zero ideal of $A/J(A)$ contains a non-zero idempotent, and hence by \cite[Proposition 1.4]{Nich75}, every ideal of $A$ not contained in $J(A)$ contains a non-zero idempotent.  It now follows from \cite[Proposition 2.1]{Nich75} that when $e$ is a primitive idempotent of a pseudocompact algebra, the algebra $eAe$ is local (see \cite[Lemma 4.2]{JohnRicardoBrauersFirstTheorem} for a different proof of this fact). 
%Any pseudocompact algebra $A$ has a propetry that idempotents can be lifted by $J(A)$ by \cite[IV. §3, Corollary 3]{Gabriel_thesis}. 
%Hence,
%from the characterization above and by \cite[Proposition 1.4]{Nich75} one obtains that if $I$ is a closed ideal of a pseudocompact algebra not contained in $J(A)$, then $I$ contains a non-zero idempotent. 
    %is an example of so called $I$-ring considered in \cite{Nich75}, i.e. an algebra in which every ideal $I\not \subseteq J(A)$ contains a non-zero idempotent.
%In particular by \cite[Proposition 2.1]{Nich75} it follows that when $e$ is a primitive idempotent of a pseudocompact algebra, then $eAe$ is local (see \cite[Lemma 4.2]{JohnRicardoBrauersFirstTheorem} for another proof of this fact). 
\end{rem}

\begin{ex}
A commutative artinian ring is a direct product of local rings (e.g. \cite[Theorem 8.7]{AM69}), and the same is true for pseudocompact algebras. Indeed, let $A$ be a commutative pseudocompact algebra.  
%Fix a complete set of primitive ortogonal idempotents $\{e_i\,|\, i\in I\}$ of $A$ (which exists, for instance, by \cite[IV. §3, Corollary  1,2]{Gabriel_thesis}). Then
By \cite[IV. §3, Corollaries  1,2]{Gabriel_thesis}, there is a set of primitive idempotents $\{e_i\,|\,i\in I\}$ in $A$ such that
%  $$
%     A=A\cdot 1=A \sum_i e_i=\prod_i Ae_ie_i= \prod_i e_i A e_i.
%  $$
$$A = \prod_{i\in I}Ae_i.$$
%where each $e_i$ is primitive, hence each $e_iA e_i$ is a local pseudocompact algebra 
By Remark \ref{rem.locality}, each $Ae_i = e_iAe_i$ is local. 
\end{ex}

\begin{ex}
From Part \ref{P2.J=0} of Proposition \ref{ssCharacterization} and Item \ref{Examples JacobsonRadical profinite groups} of Example \ref{examples Jacobson radicals}, one obtains a presumably well-known version of Maschke's Theorem for profinite groups: if $G$ is a profinite group and $k$ is a field whose characteristic does not divide the order of any continuous finite quotient group of $G$, then $k\db{G}$ is topologically semisimple.
\end{ex}

\section{Separable pseudocompact algebras and the Wedderburn-Malcev Theorem}\label{Section Separable and WM}

Separable algebras are a particularly well-behaved class of semisimple algebras whose definition becomes important when the field $k$ is not algebraically closed, and which have a literature to themselves (see e.g. \cite[Chapter 10]{Pierce} or \cite{Ford17}).
%appear as a natural strengthening of the condition that algebra is semisimple and play an important role in structure theory of associative algebra (e.g. \cite[Chapter 10]{Pierce} and \cite{Ford17}).
Having the notion of semisimple pseudocompact algebra, we are in a strong position to define and characterize separable pseudocompact algebras in a precise way.  

Having done so, we will prove a Wedderburn-Malcev Theorem for pseudocompact algebras: if $A$ is a pseudocompact algebra for which $A/J(A)$ is separable, then $A$ contains a closed subalgebra $S$ such that $A=S\oplus J(A)$ (the ``Wedderburn part''),  % Such subalgebra is called \textit{cleaving} (see, for instance, \cite{Reisel}).
and any two such subalgebras are conjugate (the ``Malcev part'').

%We adapt the definitions and basic results from \cite{Ford17,CR06}.

\subsection{Separable pseudocompact algebras} 
%and their characterization}

In what follows, by $\ctens_k$ we denote the \textit{completed tensor product} over $k$, which is defined by a universal property analogous to that of the abstract tensor product in the category of pseudocompact modules (see, \cite[Section 2]{Brumer} for the details).  Completed tensor products have similar properties to abstract tensor products,
%and they coincide in finitely generated cases 
%(see \cite[Lemma 2.1]{Brumer}). 
but if $V, W$ are pseudocompact $k$-vector spaces, then $V\otimes_k W = V\ctens_k W$ if, and only if, at least one of $V,W$ is finite dimensional \cite[Proposition 2.2]{MSZ}.% $\ctens$ differs from usual tensor product in many cases, for instance if $A=k\db{x}$ then $A\ct_k A$ is isomorphic to $k\db{x,y}\neq A\otimes_k A$.

\begin{defn}\label{defn.separability}
A pseudocompact algebra $A$ (over a field $k$) is called \textit{separable} if 
$A\ct_k E$ is a semisimple algebra over $E$ for every finite extension field $E$ of $k$.
\end{defn}

\begin{rem}
A finite dimensional algebra $A$ is usually said to be separable if $A\otimes_k E$ is semisimple for every extension field $E$ of $k$.  A finite dimensional algebra that is separable in this sense is clearly separable in our sense, and the converse is also true.  Indeed, it is enough 
%([[check the case $F=D$ is a division ring as well!!]]) 
to check this for finite dimensional simple algebras.  By the usual characterization of separable algebras (see for example \cite[Theorem 4.5.7]{Ford17}), if $M_n(D)$ is not separable then $D$ is a division algebra whose centre $Z$ is a non-separable finite field extension of $k$ (whose characteristic must therefore be a prime number $p$).   We will check that $Z\otimes_k M_n(D)$ is not semisimple.  Let $S:k$ be the separable closure of the extension $Z:k$. Hence $Z:S$ is a (non-trivial) purely inseparable extension, so for any $\alpha\in Z\setminus S$ there is $n\in \mathbb N$ such that $\alpha^{p^n}\in S$ \cite[Theorem V.6.4]{Hungerford}.  One checks that $x = \alpha \otimes 1-1\otimes \alpha$ is nilpotent ($x^{p^n}=0$) in $Z\otimes_S Z$. This implies (by the Wedderburn-Artin theorem) that $J(Z\otimes_S Z)\neq 0$. But now the surjection $Z\otimes_k Z \twoheadrightarrow Z\otimes_S Z$ implies by Corollary \ref{radical surjective} that $J(Z\otimes_k Z)\neq 0$. %(see for instance \cite[Corollary 2.4]{IM17}). 
An element $0\neq y \in J(Z\otimes_k Z)$ has torsion, and hence the element $y\otimes 1 \in (Z\otimes_k Z)\otimes_Z D\iso Z\otimes_k D$ has torsion, so that $J(Z\otimes_k D)\neq 0$.  Finally 
$$
J(Z\otimes_k M_n(D))=J(M_n(Z\otimes_k D))=M_n(J(Z\otimes_k D))\neq 0.
$$
%So our definition of separable algebra coincides with the normal one on the class of finite dimensional algebras.
\end{rem}

%We sketch the argument: 

% Let $D$ be a finite dimensional $k$-division algebra with center $Z$ and suppose that $E\otimes_k M_n(D)$ is semisimple for every finite field extension $E$ of $k$.  We claim that $Z$ is a separable extension of $k$.  If not, then $\overline{k}\otimes_k Z$ has a non-zero nilpotent element $x$.  Let $E$ be a finite extension of $k$ containing $x$.

% if $A$ is a finite dimensional separable algebra then by [REF!!!] it is a product of $M_n(D)$ with $D$ a finite division $k$-algebra whose center $Z$ is a separable field extension of $k$.  A maximal subfield $F$ of $D$ splits $D$ as a $Z$-algebra, meaning that $F\otimes_Z D\iso M_n(F)$ for some $n$.  As $Z$ is a separable extension of $k$ it is simple and a splitting field $E$ for the minimal polynomial of the generator is such that $E\otimes_k Z \iso \prod E$ by the Chinese Remainder Theorem, again because $Z:k$ is a separable extension.  Now
% $$EF\otimes_k D \iso EF\otimes_k\otimes Z\otimes_Z D\iso \prod EF\otimes_Z D \iso \prod M_n(EF).$$

% We allow only finite extensions $E$ of $k$ because infinite extensions are not pseudocompact as $k$-algebras.  If $A$ is a finite dimensional algebra, a more common definition of separability is that $A\otimes_k E$ be semisimple for every extension field $E$ of $k$.  However, by [some ref] this is the case if, and only if, $A\otimes_k E$ is semisimple for every finite extension of $k$, and so the two definitions coincide in this case.

%To characterize separable pseudocompact algebra we will need some concepts, which naturally 
We generalize some further concepts from finite dimensional to pseudocompact algebras. A \emph{pseudocompact derivation} of the pseudocompact algebra $A$ is a pseudocompact $A$-bimodule $T$ together with a continuous linear map $d:A\to T$ having the property that
$$d(ab) = ad(b) + d(a)b.$$
An \emph{inner pseudocompact derivation of $A$} is a pseudocompact derivation of the form 
%pseudocompact $A$-bimodule $T$ together with a map $A\to T$ given by 
$a\mapsto ua - au$, for some fixed $u\in T$ (such a map is easily checked to be a derivation).  Consider $A\ctens_k A$ as an $A$-bimodule as follows:
$$a\cdot(b\ctens c) := ab\ctens c\,,\quad (b\ctens c)\cdot a := b\ctens ca.$$
Observe that $A\ctens_k A$ is the free $A$-bimodule of rank $1$ (freely generated by $1\ctens 1$).  This done, we can treat the multiplication of $A$ as a continuous $A$-bimodule homomorphism $m:A\ctens_k A \to A$ given on pure tensors as $b\ctens c\mapsto bc$.  Finally, a \emph{separability idempotent for $A$} (if it exists) is an element $p\in A\ctens_k A$ with the following properties:
$$m(p) = 1\,,\quad ap = pa \quad  \forall a\in A.$$

\begin{theorem}\label{Theorem separable characterization}
The following statements concerning a pseudocompact algebra $A$ over a field $k$ are equivalent:
\begin{enumerate}[label=\textnormal{(\arabic*)}]
\item $A$ is a separable algebra; \label{T1.sepAlg}
\item $A\iso \prod M_{n_i}(\Delta_i)$ with each $\Delta_i$ a finite $k$-division algebra whose center is a separable field extension of $k$; \label{T1.prodSepMatrices}
\item $A$ is an inverse limit of separable finite-dimensional algebras; \label{T1.invLimSep}
\item $A$ is projective as an $A$-bimodule; \label{T1.AisProj}
\item the multiplication map $m: A\ctens_k A\to A$ splits as a homomorphism of $A$-bimodules;  \label{T1.mSplits}
\item $A$ has a separability idempotent $p$; \label{T1.sepIdemp}
\item every generalized pseudocompact derivation $d:A\to T$ is inner; \label{T1.derInner}
\item The universal derivation $f:A\to \tn{Ker}(m)$ defined by $a\mapsto 1\otimes a-a\otimes 1$ is inner.
\label{T1.univInner}
%\item Something with $\textrm{Ext}^n(-,-)$? \label{T1.Ext}
\end{enumerate}
\end{theorem}

\begin{proof}

[$\ref{T1.sepAlg}\Rightarrow\ref{T1.prodSepMatrices}$] If $A$ is separable then in particular it is semisimple, so by Proposition \ref{ssCharacterization} a product of matrix algebras $A = \prod M_{n_i}(\Delta_i)$, where each $\Delta_i$ is a finite dimensional $k$-division algebra.  If some $\Delta_i$ did not have the form given in \ref{T1.prodSepMatrices} then by \cite[Proposition 10.7]{Pierce}, $M_{n_i}(\Delta_i)$, hence also $A$, would not be separable.

[$\ref{T1.prodSepMatrices}\Rightarrow \ref{T1.invLimSep}$] $\prod_{I} M_{n_i}(\Delta_i) \iso \invlim_{F\subseteq I, |F|<\infty} \prod_{F}M_{n_i}(\Delta_i)$ and each of these finite products is separable, by \cite[Proposition 10.7]{Pierce}.

%A finite product of matrices $ M_n(E)$ with $E$ a separable extension field of $k$ is separable by the finite dimensional version of the result, so this is immediate.

[$\ref{T1.invLimSep}\Rightarrow \ref{T1.sepAlg}$] Write $A = \invlim A_i$ with each $A_i$ finite dimensional and separable.  For a finite extension field $E$ of $k$ we have
$$A\ctens_k E = \invlim (A_i\ctens_k E)$$
by \cite[Section 2]{Brumer}.  Each $A_i\ctens_k E$ is semisimple so $A\ctens_k E$ is semisimple by Proposition \ref{ssCharacterization}.  Thus $A$ is separable.

% [$\ref{T1.sepAlg}\Leftrightarrow\ref{T1.invLimSep}$] Write $A = \invlim \{A_i, \varphi_{ij}\}$ with each $A_i$ finite dimensional and each $\varphi_i$ surjective.  For a finite extension field $E$ of $k$ we have
% $$A\ctens E = \invlim (A_i\ctens E).$$
% If each $A_i\ctens E$ is semisimple then $A\ctens E$ is semisimple by Proposition \ref{ssCharacterization}, while if $A\ctens E$ is semisimple then each $A_i\ctens E$ is semisimple, being a quotient of $A\ctens E$.  The required equivalence follows.

% [$\ref{T1.invLimSep}\Rightarrow\ref{T1.sepAlg}$]
% Suppose $(2)$ and let $E$ be any finite field extension of $k$. Then $A\ct_k E=(\invlim A_i)\ct_k E=\invlim (A_i\otimes E)$
% but each $A_i\otimes E$ is semisimple by assumption. Hence $A\ct_k E$ is a semisimple pseudocompact algebra by Proposition \ref{ssCharacterization}. This implies that $A$ is separable.

% [$\ref{T1.sepAlg}\Rightarrow\ref{T1.invLimSep}$]
% Assume that $A=\invlim A_i$ (in which all the maps are surjective) is separable. Hence for any finite field extension $E$ of $k$ we have that 
% $A\ct_k E=(\invlim A_i)\ct_k E=\invlim (A_i\otimes_k E)$ is semisimple. Hence, by Proposition 2.3., each $A_i\otimes_k E$ is a quotient (?) os semisimple algebra, i.e. each $A_i\otimes_k E$ is semisimple. It implies (via the comment above) that each $A_i$ is separable.

[$\ref{T1.invLimSep}\Rightarrow\ref{T1.AisProj}$]
Write $A=\invlim A_i$ with each $A_i$ finite dimensional and separable, with surjective maps. By the finite dimensional version of this result (see, for instance, \cite[Section 10.2]{Pierce}) each $A_i$ is a projective $A_i$-bimodule.  Recall that an $A$-bimodule is the same thing as a left $A\ctens_k A^{\tn{op}}$-module -- the equivalence may be checked as with abstract (bi)modules.
%[I THINK THIS CAN BE CHECKED DIRECTLY AND JUST STATED AS FACT HERE]. %
We have $A\ctens_k A^{\tn{op}} = \invlim A_i\ctens_k A_i^{\tn{op}}$ and the maps remain surjective, hence $A$ is a projective $A\ctens_k A^{\tn{op}}$-module by \cite[Corollary 3.3]{Brumer}

%[$\ref{T1.mSplits}\Rightarrow\ref{T1.AisProj}$] As $A\ctens A$ is a free $A$-bimodule, this is immediate.

[$\ref{T1.AisProj}\Rightarrow\ref{T1.mSplits}$] This is immediate: if $A$ is projective as a bimodule then any continuous bimodule homomorphism onto $A$ splits, and in particular $m$ splits.

[$\ref{T1.mSplits}\Rightarrow\ref{T1.sepIdemp}$] 
%For both this and the next implication I used:
%https://qchu.wordpress.com/2016/03/27/separable-algebras/
Let $\gamma : A\to A\ctens_k A$ split $m$ and write $p = \gamma(1)$.  Then
$$m(p) = m\gamma(1) = 1$$
$$ap = a\gamma(1) = \gamma (a\cdot 1) = \gamma(1\cdot a) = \gamma(1)a = pa$$
since $\gamma$ is a bimodule hom.

[$\ref{T1.sepIdemp}\Rightarrow\ref{T1.invLimSep}$] This is essentially immediate.  Given an open ideal $I$ denote by $m_I:A/I\otimes_k A/I \to A/I$ the multiplication map, and let $p$ be a separability idempotent of $A$.  The commutativity of the square
$$\xymatrix{
A\ctens_k A \ar[r]^{m} \ar[d]_{\pi\ctens \pi} & A \ar[d]^{\pi} \\
A/I\otimes_k A/I\ar[r]_-{m_I} & A/I
}$$
shows that the element $p_I = (\pi\ctens \pi)(p)$ is a separability idempotent of $A/I$:
$$m_I(p_I) = m_I(\pi\ctens \pi)(p) = \pi m(p) = \pi(1) = 1,$$
\begin{align*}
(m_I\ctens 1)((a+I)\ctens p_I) & = (m_I\ctens 1)(\pi\ctens \pi\ctens \pi)(a\ctens p) \\
& = (\pi\ctens \pi )(m\ctens 1)(a\ctens p) \\
& = (\pi\ctens \pi)(1\ctens m)(p\ctens a) \\
& = (1\ctens m_I)(p_I\ctens (a+I)).
\end{align*}

% [$\ref{T1.sepIdemp}\Rightarrow\ref{T1.mSplits}$] Define a map $\gamma: A \to A\ctens A$ by $1\mapsto p$.  This defines a unique continuous map of left $A$-modules given by $\gamma(a) = ap$.  We must only check it's also a right $A$-module hom.  But 
% $$\gamma(ab) = abp = apb = \gamma(a)b$$
% by the second property of $p$.  Further
% $$m\gamma(1) = m(p) = 1$$
% and so $\gamma$ splits $m$, as required.

[$\ref{T1.sepIdemp}\Rightarrow\ref{T1.derInner}$] %If we could write elements of the completed tensor product as products of pure tensors this would be much easier -- can we? I have tried to avoid this using coalgebra type notation.  Maybe good to prove something for coalgebras and dualize for the PC result?  Anyway, notation:
%$$m : A\ctens A\to A\qquad (\hbox{multiplication map})$$
By definition, an $A$-bimodule $T$ comes equipped with actions
$$\lambda : A\ctens T \to T\,,\, \mu: T\ctens A \to T
$$
%\qquad (\hbox{action maps})$$
which satisfy the following identities:
$$\lambda(1_A\ctens\lambda) = \lambda(m\ctens 1_T)\,,\quad \mu(\mu\ctens 1_A) = \mu(1_T\ctens m)\,,\quad \mu(\lambda\ctens 1_A) = \lambda(1_A\ctens \mu).$$
Let $p$ be a separability idempotent.  The properties of $p$ translate as
$$m(p) = 1\,,\, (m\ctens 1)(a\ctens p) = (1\ctens m)(p\ctens a), \quad a\in A.$$
Finally $d$ being a derivation translates into
$$dm = \lambda(1\ctens d) + \mu(d\ctens 1).$$
Applying $\mu(d\ctens 1)$ to the separability idempotent equation we get
\begin{align*}
\mu(d\ctens 1)(m\ctens 1)(a\ctens p) & = \mu(d\ctens 1)(1\ctens m)(p\ctens a) \\
\mu(dm\ctens 1)(a\ctens p) & = \mu(d\ctens m)(p\ctens a) \\
\mu((\lambda(1\ctens d) + \mu(d\ctens 1))\ctens 1)(a\ctens p) & = \mu(d\ctens 1)(1\ctens m)(p\ctens a) \\
\mu\lambda(1\ctens d)(a\ctens p) + \mu(\mu(d\ctens 1))\ctens 1)(a\ctens p) & = \mu(d\ctens 1)(1\ctens m)(p\ctens a).
\end{align*}
Analyzing the second term we obtain:
\begin{align*}
\mu(\mu(d\ctens 1))\ctens 1)(a\ctens p) 
& = \mu(\mu\ctens 1)(d\ctens 1\ctens 1)(a\ctens p) \\
& = \mu(1\ctens m)(d(a)\ctens p) \\
& = \mu(d(a)\ctens m(p)) \\
& = d(a).
\end{align*}

Define $u\in T$ to be the element $\mu(d\ctens 1)(p)$.  The first term yields
\begin{align*}
\mu\lambda(1\ctens d)(a\ctens p) 
& = \mu(\lambda\ctens 1)(1\ctens d\ctens 1)(a\ctens p) \\
& = \lambda(1\ctens \mu)(a\ctens(d\ctens 1)(p)) \\
& = \lambda(a\ctens \mu(d\ctens 1)(p)) \\
& = \lambda(a\ctens u)
\end{align*}
while the third term yields
\begin{align*}
\mu\lambda(1\ctens d)(a\ctens p) 
& = \mu(1\ctens m)(d\ctens 1\ctens 1)(p\ctens a) \\
& = \mu(\mu\ctens 1)((d\ctens 1)(p)\ctens a) \\
& = \mu(\mu(d\ctens 1)(p)\ctens a) \\
& = \mu(u\ctens a).
\end{align*}
Thus $d(a) = \mu(u\ctens a) - \lambda(a\ctens u)$, or in compact notation, $d(a) = ua - au$.
Therefore $d$ is inner.

[$\ref{T1.derInner}\Rightarrow\ref{T1.univInner}$] The universal derivation is pseudocompact, so this is immediate.

[$\ref{T1.univInner}\Rightarrow\ref{T1.sepIdemp}$] This is identical to the finite version.  The universal derivation is inner, so write $f(a) = ua - au$ for some $u\in \tn{Ker}(m)$ and define $p = 1\ctens 1 - u$.  Then $p$ is a separability idempotent:
$$m(p) = 1 - m(u) = 1.$$
$$ap-pa = a\ctens 1 - au - 1\ctens a + ua = f(a) - f(a) = 0\quad\forall a\in A.$$

\end{proof}

\begin{corol}\label{corol quotient of sep is sep}
Let $A$ be a separable pseudocompact algebra and $I$ a closed ideal of $A$.  Then $A/I$ is separable.
\end{corol}

\begin{proof}
By Theorem \ref{Theorem separable characterization}, $A$ is a direct product of algebras $\prod_{i\in X}M_{n_i}(\Delta_i)$, where $\Delta_i$ is a $k$-division algebra whose center is a separable extension of $k$.  Treating the ideal $I$ as a left $A$-module, it has a unique decomposition of the form $I = \prod_{i\in X}1_{M_{n_i}(\Delta_i)}\cdot I = \prod_{i\in X}(M_{n_i}(\Delta_i)\cap  I)$ -- this can be checked directly, or it follows from \cite[Proposition 4.3]{JohnRicardoBrauersFirstTheorem}.  But $M_{n_i}(\Delta_i)\cap  I$ is an ideal of $M_{n_i}(\Delta_i)$ and is hence either $M_{n_i}(\Delta_i)$ or $0$.  It follows that $I = \prod_{i\in Y}M_{n_i}(\Delta_i)$, for some subset $Y$ of $X$.  The algebra $A/I$ is thus isomorphic to $\prod_{i\in Y\setminus X}M_{n_i}(\Delta_i)$, which is separable by Theorem \ref{Theorem separable characterization}.
\end{proof}

\subsection{The Wedderburn splitting theorem}

If $A$ is a pseudocompact algebra, then by Propositin \ref{ssCharacterization}, $A/J(A)$ is a topologically semisimple pseudocompact algebra.  Even for finite dimensional algebras, the canonical projection $A\to A/J(A)$ need not split as an algebra homomorphism.  But when $A/J(A)$ is \emph{separable} in the sense of Definition \ref{defn.separability}, it always does -- this is a version for pseudocompact algebras of the famous Wedderburn Splitting Theorem.  To our knowledge, a direct proof of this result does not exist in the literature.  The existence of the splitting when $A/J^2(A)$ is finite dimensional is a result of Curtis \cite[Theorem 1]{Curtis54}.  The result in full generality has been proved for coalgebras by Abe \cite[Theorem 2.3.11]{Abe}, and so follows for pseudocompact algebras by duality.  We present a complete proof of this important result for pseudocompact algebras (Theorem \ref{theorem Wedderburn splitting}).  Our proof is essentially dual to Abe's, but we think it is worth presenting for several reasons: the proof becomes accessible to the reader not familiar with the language of coalgebras; Abe's proof itself employs duality, and dualizing twice seems rather unnatural; finally, because Abe passes rather quickly over a subtle point that we feel benefits from more attention (namely, the dual of Lemma \ref{lemma splitting sends ideal to ideal}, which is treated as self-evident).

\begin{lemma}\label{lemma splitting sends ideal to ideal}
Let $A$ be a pseudocompact algebra, let $I$ be a closed ideal of $A$, and suppose that the canonical projection from $A$ to $A/J(A) = A/J$ has a splitting $s : A/J\to A$ as an algebra homomorphism.  Then $s(I+J)\subseteq I$.
\end{lemma}

\begin{proof}
First note that it sufficient to check this for $I$ open, because if the claim holds for open ideals and $I$ is closed, then 
$$s(I+J) = s\Big(\bigcap_{L\lhd_O A}(I+L)+J\Big)\subseteq
\bigcap_{L\lhd_O A}s\left((I+L)+J\right)\subseteq \bigcap_{L\lhd_O A} I+L = I.$$
Suppose from now on that $I$ is open.  Then for some $n$, $J^n(A)\subseteq I$.  The closed subspace $(I+J)/J$ is an ideal of $A/J$ and hence, by Proposition \ref{ssCharacterization}, is a direct product of matrix algebras $M_n(\Delta)$.  Since $s$ is an algebra homomorphism, it is sufficient to show that $s$ sends the identity of each such matrix algebra inside $I$.  So fix such a factor $M_n(\Delta)$ and $x\in I$ such that $x+J = 1_{M_n(\Delta)}$.  By the definition of $s$, $s(x+J) = x+j$, for some $j\in J$.  Using that $x+J$ is idempotent, we have
$$s(x+J) = s((x+J)^n) = s(x+J)^n = (x+j)^n.$$
Expanding $(x+j)^n$, every term is an element of $I$ since $I$ is an ideal, except the final term $j^n$, which is an element of $J^n\subseteq I$.  So $s(x+J)\in I$ as required.
\end{proof}

\begin{theorem}\label{theorem Wedderburn splitting}
Let $A$ be a pseudocompact algebra such that $A/J$ is separable.  The canonical projection $A\to A/J$ splits continuously as an algebra homomorphism.
\end{theorem}

\begin{proof}
We consider the set $\mathcal{F}$ of pairs $(I,s)$, where $I$ is a closed ideal of $A$ contained in $J$ and $s: A/J \to A/I$ is a splitting of the natural projection $A/I \to A/J$.  Then $\mathcal{F}$ is non-empty, since it contains $(J, \tn{id})$.  Order $\mathcal{F}$ by declaring $(I,s)\leqslant (I',s')$ whenever $I\subseteq I'$ and the diagram
$$\xymatrix{
& A/I\ar[dd] \\
A/J\ar[ur]^{s}\ar[dr]_{s'} & \\
& A/I'}$$
commutes, where the vertical map is the canonical projection.  A chain $\mathcal{C}$ in $\mathcal{F}$ is by definition a totally ordered chain of closed ideals $I$ together with a map of inverse systems
$$\{s : A/J \to A/I \,|\, (I,s) \in \mathcal{C}\},$$
which yields a unique map $\invlim s : A/J\to A/\bigcap_{\mathcal{C}}I$.  Then $(\bigcap_{\mathcal{C}}I, \invlim s)$ is a lower bound for $\mathcal{C}$ and so by Zorn's Lemma, $\mathcal{F}$ contains a minimal element $(I_0, s_0)$.  Our task is to prove that $I_0=0$.  Suppose it were not.  Then there is an open ideal $I$ of $A$ such that $I\cap I_0$ is properly contained in $I_0$.  We will find a splitting $s' : A/J \to A/(I\cap I_0)$ such that $(I\cap I_0, s') < (I_0, s_0)$, contradicting the minimality of $(I_0, s_0)$ and completing the proof. 

We construct two splittings of the canonical projection map $q : A/(I+I_0)\to A/(I+J)$.  
\begin{itemize}
    \item First let $s$ be an algebra splitting of $A/I\to A/(I+J) = (A/I)/J(A/I)$.  This splitting exists by the Wedderburn splitting theorem for finite dimensional algebras \cite[Theorem 72.19]{CR62}, which applies since $A/(I+J)$ is separable by Corollary \ref{corol quotient of sep is sep}.  Composing $s$ with the canonical projection $v : A/I \to A/(I+I_0)$ we obtain a splitting $vs$ of $q$.  
    
    %\item Second, the map $s_0 : A/J \to A/I_0$ sends $(I+J)/J$ into $(I+I_0)/I_0$ by Lemma \ref{lemma splitting sends ideal to ideal}, and hence induces a map $\overline{s_0} : A/(I+J)\to A/(I+I_0)$, which is another splitting of $q$.
    % \item To minimize notational confusion, write $B = A/I_0$ and $\tn{Im}(s_0) = \Sigma$, a subalgebra complementing $J(B)$ in $B$ and treat $s_0$ as the inclusion $\Sigma \hookrightarrow \Sigma \oplus J(B)$.  The kernel of the composition
    % $$\Sigma \hookrightarrow \Sigma\oplus J(B) \to (\Sigma \oplus J(B))/(I+I_0)$$
    % is $\Sigma\cap (I+I_0)$, and hence the map factors through $\Sigma/(\Sigma\cap (I+I_0))\iso (\Sigma + I + I_0)/(I+I_0)$ like this:
    % $$\xymatrix{
    % \Sigma \ar@{^{(}->}[r]^{s_0}\ar[d] & \Sigma \oplus J(B)\ar[d] \\
    % (\Sigma + I + I_0)/(I+I_0)\ar[r]_{\overline{s_0}} & (\Sigma\oplus J(B))/(I+I_0)
    % }$$
    
    \item Second, since $A/J$ is semisimple, by Part \ref{P2.prodFullM} of Proposition \ref{ssCharacterization} the kernel $(I+J)/J$ of the canonical projection $\pi : A/J \to A/(I+J)$ has a unique complement $S$ in $A/J$.  Denote by $\iota : A/(I+J)\to A/J$ the (non-unital) splitting of $\pi$, so that $S = \tn{Im}(\iota)$.  If $p : A/I_0 \to A/(I+I_0)$ is the canonical projection, then the map $ps_0\iota$ is another splitting of $q$.
    
    % %Second, since $A/J$ is semisimple, by [REF] the canonical projection $\pi : A/J \to A/(I+J)$ splits via a non-unital algebra map $\iota : A/(I+J)\to A/J$.  Denoting by $p : A/I_0 \to A/(I+I_0)$ the canonical projection, the map $ps_0\iota$ is another splitting of $q$.
\end{itemize}
The Malcev uniqueness theorem for finite dimensional algebras implies that there is an element $a\in {J(A/(I+I_0))}$ such that
$$(1+a)vs(1+a)^{-1} = ps_0\iota.
%\overline{s_0}
$$
%$$(1+a)\overline{s}(1+a)^{-1} = ps_0\iota.$$
Let $a'$ be an element of $J(A/I)$ that projects onto $a$ and replace $s$ with $(1+a')s(1+a')^{-1}$.  We may thus assume from now on that % $vs = \overline{s_0}$.
$vs = ps_0\iota$.  
Put these maps together in the following diagram:
$$\xymatrix{
& A/J\ar[dl]_{\pi}\ar@/^{20pt}/[ddr]^{s_0} & \\
A/(I+J)\ar[d]_{s} & A/(I\cap I_0)\ar[dl]\ar[dr] & \\
A/I\ar[dr]_{v} && A/I_0\ar[dl]^{p} \\
& A/(I+I_0) &
}$$
The lower square is a pullback diagram, and so if the outer shape commutes, then there is a unique map $s' : A/J\to A/(I\cap I_0)$ making the diagram commute, and this $s'$ will be the splitting we require.  In order to check that the outer shape commutes, write $A/J = \tn{Im}(\iota)\times (I+J)/J$.  Given $x = \iota(y)$ in the first factor, we have
$$vs\pi(x) = vs\pi\iota(y) = ps_0\iota\pi\iota(y) = 
ps_0\iota(y) = ps_0(x).$$
An element $x$ of the second factor is sent to $0$ by $\pi$, so we need to check that $ps_0(I+J)\subseteq I+I_0$.  But $s_0(I+J)\subseteq I+I_0$ by Lemma \ref{lemma splitting sends ideal to ideal}, so we are done.
\end{proof}

%\subsection{Conjugate cleavings}
\subsection{The Malcev uniqueness theorem}

Now that we have the ``Wedderburn part'' of the fundamental Wedderburn-Malcev Theorem, we turn to the ``Malcev part'', which says that the splitting of Theorem \ref{theorem Wedderburn splitting} is unique up to conjugation by a unit of the form $1+\omega$, where $\omega$ is an element of the radical of $A$.  The Malcev uniqueness theorem for pseudocompact algebras was proved by Eckstein \cite[Theorem 17]{Eckstein}.  We provide a simple proof that closely mimics the argument for finite dimensional algebras (cf.\ for instance \cite[Theorem 72.19]{CR62}):
 
%The Wedderburn splitting theorem, which asserts at least one splitting of the algebra homomorphism $A\to A/J(A)$ when $A/J(A)$ is a separable algebra, can be be extended to pseudocompact algebras via the dual result for coalgebras (e.g., \cite[Theorem 2.3.11]{Abe}).  We extend here the other half of the Wedderburn--Malcev Theorem:

\iffalse
\begin{theorem}[Malcev Uniqueness Theorem]
Suppose that the algebra $A/J$ is separable and let $s,t$ be splittings of the projection $\pi : A\to A/J$.  There is $\omega\in J$ such that
$$\tn{Im}(s) = {}^{1-\omega}\tn{Im(t)}.$$
\end{theorem}

\begin{proof}
Define the structure of $A/J$-bimodule on $J$ by 
$$x\cdot j := s(x)j\,,\quad j\cdot x := jt(x)\quad(x\in A/J, j\in J).$$
The continuous function $d : A/J\to J$ given by $d(x) = s(x) - t(x)$ is well-defined, because
$$\pi d(x) = \pi s(x) - \pi t(x) = x-x = 0,$$
so that $d(x)\in J$.  Furthermore, $d$ is a derivation:
\begin{align*}
d(xy) & = s(xy) - s(x)t(y) + s(x)t(y) - t(xy) \\
      & = s(x)(s(y) - t(y)) + (s(x) - t(x))t(y) \\
      & = x\cdot d(y) + d(x)\cdot y.
\end{align*}
By Theorem \ref{Theorem separable characterization}, there is $\omega \in J$ such that
$$d(x) = s(x) - t(x) = x\cdot \omega - \omega\cdot x = s(x)\omega - \omega t(x)$$
for all $x\in A/J$, so that
$$s(x)(1-\omega) = (1-\omega)t(x).$$
The element $1-\omega$ is invertible by Proposition \ref{radicalChar} and so $\tn{Im}(s) = {}^{1-\omega}\tn{Im(t)}$ as required.
\end{proof}

\fi 

\begin{theorem}[Malcev Uniqueness Theorem]\label{Malcev Uniqueness}
Suppose that the algebra $A/J$ is separable, and let $S_1$, $S_2$ be two closed subalgebras of $A$ such that 
$
    S_1\oplus J(A)=A=S_2\oplus J(A).
$
There is $\omega\in J$ such that
$$S_1 =(1-\omega)S_2(1-\omega)^{-1}.$$
%$$S_1(1-\omega) =(1-\omega)S_2.$$
\end{theorem}

\begin{proof}
%The proof mimics the classical one (see, for instance, \cite[Chapter 10]{Pierce}).
Let $s_1, s_2:A/J\to A$ be splittings of the canonical projection $A\to A/J$ with images $S_1, S_2$, respectively.  We define an $A/J$-bimodule structure on $J$ by 
$$x\cdot j := s_1(x)j\,,\quad j\cdot x := js_2(x)\quad(x\in A/J, j\in J).$$
The continuous function $d : A/J\to J$ given by $d(x) = s_1(x) - s_2(x)$ is well-defined, because
$$\pi d(x) = \pi s_1(x) - \pi s_2(x) = x-x = 0,$$
so that $d(x)\in J$.  Furthermore, $d$ is a derivation:
\begin{align*}
d(xy) & = s_1(xy) - s_1(x)s_2(y) + s_1(x)s_2(y) - s_2(xy) \\
      & = s_1(x)(s_1(y) - s_2(y)) + (s_1(x) - s_2(x))s_2(y) \\
      & = x\cdot d(y) + d(x)\cdot y.
\end{align*}
By Theorem \ref{Theorem separable characterization}, there is $\omega \in J$ such that
$$d(x) = s_1(x) - s_2(x) = x\cdot \omega - \omega\cdot x = s_1(x)\omega - \omega s_2(x)$$
for all $x\in A/J$, so that
$$s_1(x)(1-\omega) = (1-\omega)s_2(x).$$
The element $1-\omega$ is invertible by Proposition \ref{radicalChar} and so 
$$\tn{Im}(s_1) = (1-\omega)\tn{Im}(s_2)(1-\omega)^{-1}$$
%$$\tn{Im}(s_1)(1-\omega) = (1-\omega)\tn{Im}(s_2)$$ 
as required.
\end{proof}

\section*{Acknowledgements}
We thank Mark Kleiner for stimulating discussion and a number of important comments on an earlier version of the manuscript. The first author was partially supported by FAPESP grant
2018/23690-6. The second author was partially supported by CNPq Universal Grant
402934/2021-0.

%\bibliographystyle{alpha} 
%\bibliographystyle{plain}
%\bibliography{mref}

\end{document}